\documentclass[10pt,a4paper]{article}
\usepackage{a4wide}
\usepackage{amsfonts}
\usepackage{amsmath}
\usepackage{amssymb}
\usepackage{amsthm}
\usepackage[T1]{fontenc}
\usepackage{verbatim}

\theoremstyle{plain}
\newtheorem{theorem}{Theorem}[section]
\newtheorem{proposition}[theorem]{Proposition}
\newtheorem{corollary}[theorem]{Corollary}
\newtheorem{lemma}[theorem]{Lemma}

\theoremstyle{definition}
\newtheorem{definition}[theorem]{Definition}
\newtheorem{example}[theorem]{Example}
\newtheorem{remark}[theorem]{Remark}
\newtheorem{question}[theorem]{Question}

\newcommand{\integers}{\mathbb{Z}}
\newcommand{\posintegers}{\mathbb{N}}
\newcommand{\End}{\operatorname{End}}
\newcommand{\diag}{\operatorname{diag}}
\newcommand{\Id}{\operatorname{Id}}
\newcommand{\Nil}{\operatorname{Nil}}
\newcommand{\U}{\operatorname{U}}
\newcommand{\soc}{\operatorname{soc}}

\title{Generalizing $\pi$-regular rings}
\date{May 11, 2014}
\author{Peter Danchev\\Department of Mathematics\\University of Plovdiv, Bulgaria\\{\small e-mail: pvdanchev@yahoo.com}\and
Janez \v{S}ter\\Department of Mathematics\\University of Ljubljana, Slovenia\\{\small e-mail: janez.ster@fmf.uni-lj.si}}

\begin{document}

\maketitle

\begin{abstract}
We introduce the class {\it weakly nil clean} rings, as rings $R$ in which for every $a\in R$ there exist
an idempotent $e$ and a nilpotent $q$ such that $a-e-q\in eRa$. Every weakly nil clean ring is exchange.
Weakly nil clean rings contain $\pi$-regular rings as a proper subclass,
and these two classes coincide in the case of central idempotents. Every weakly nil clean ring of bounded index and every weakly nil clean PI-ring is strongly $\pi$-regular.
The center of a weakly nil clean ring is strongly $\pi$-regular, and consequently,
every weakly nil clean ring is a corner of a clean ring. These results extend Azumaya \cite{azumaya}, McCoy \cite{mccoy}, and the second author \cite{ster3} to a wider class of rings
and provide partial answers to some open questions in \cite{diesl} and \cite{ster3}. Some other properties are also studied and several examples are given.
\end{abstract}

\section{Introduction}

Let $R$ be an associative ring with unity.
We say that $R$ is {\it $\pi$-regular} if for every $a\in R$ there exists a positive integer $n$ such that $a^n\in a^nRa^n$.
The very first mention of this notion dates back to 1939, when it was introduced by McCoy in \cite{mccoy}, as a generalization of von Neumann regular rings.
Examples of $\pi$-regular rings, besides von Neumann regular ones, are Artinian rings and perfect rings.
While in McCoy's paper, most attention was given to $\pi$-regular rings that are commutative, the study of the general case was continued through the next decades by other authors.
Azumaya, Tominaga, and Yamada (\cite{azumaya}, \cite{tominaga}, \cite{tominagayamada}) studied $\pi$-regular rings of bounded index of nilpotence.
Menal \cite{menal} studied $\pi$-regular rings with primitive factors Artinian.
Many other results can be found in the literature, e.g.~\cite{fuchsrangaswamy}, \cite{fishersnider}, \cite{hirano}, \cite{badawi}, \cite{kimlee} etc.

Along with this theory, a similar notion of {\it strongly $\pi$-regular} rings was introduced and studied. A ring $R$ is called strongly $\pi$-regular
if for every $a\in R$ there exists a positive integer $n$ such that $a^n\in a^{n+1}R$. The notion was first considered by Kaplansky \cite{kaplansky},
and later by Azumaya \cite{azumaya} who called such rings {\it right $\pi$-regular}. Dischinger \cite{dischinger} proved that the notion is left-right symmetric.
Consequently, every strongly $\pi$-regular ring is $\pi$-regular, and the notions coincide in the case when the ring is commutative, or more general,
when idempotents in the ring are central.
Von Neumann regular rings in general are not strongly $\pi$-regular, while Artinian rings and perfect rings are, as well as algebraic algebras over a field.
Azumaya \cite{azumaya} proved that every $\pi$-regular ring with bounded index of nilpotence is strongly $\pi$-regular.
Ara \cite{ara4} proved that every strongly $\pi$-regular ring has stable range one. A module over a ring satisfies the Fitting lemma if and only if its endomorphism ring
is strongly $\pi$-regular (see \cite{nicholson3}). For more information about $\pi$-regular and strongly $\pi$-regular rings, we refer the reader to \cite{tuganbaev}.

We call a ring $R$ an {\it exchange} ring if for every $a\in R$ there exists an idempotent $e\in Ra$ such that $1-e\in R(1-a)$,
and $R$ is called {\it clean} if every element in $R$ can be written as a sum of an idempotent and an invertible element.
Clean rings are exchange, but the converse does not hold (see \cite{nicholson}, \cite{camilloyu}).

Stock \cite{stock} proved that all $\pi$-regular rings are exchange, and Burgess and Menal \cite{burgessmenal}, and Nicholson \cite{nicholson3} proved that every strongly $\pi$-regular
ring is clean, but none of these two implications is reversible. In \cite{ster}, the second author introduced {\it weakly clean} rings,
as rings $R$ in which for every $a\in R$ there exist an idempotent $e$ and a unit $u$ such that $a-e-u\in(1-e)Ra$. As shown in \cite{ster} and \cite{ster3},
every weakly clean ring is exchange and every clean ring is weakly clean, but none of these two implications is reversible. Furthermore, every $\pi$-regular ring
is weakly clean \cite{ster3}.

In this paper we introduce a class of rings which is a proper subclass of weakly clean rings as defined in \cite{ster}, but still contains all $\pi$-regular rings.
These rings, which will be called {\it weakly nil clean} rings, are defined as rings $R$ in which for every $a\in R$ there exist
an idempotent $e$ and a nilpotent $q$ such that $a-e-q\in eRa$. Examples of such rings include, besides all $\pi$-regular rings,
also semiperfect rings with nil Jacobson radical (which in general need not be $\pi$-regular),
finite direct products of $\pi$-regular rings, and upper triangular matrix rings over $\pi$-regular rings.

We show that weakly nil clean rings are closely related to $\pi$-regularity. For example, in the case of central idempotents both notions coincide.
Moreover, we generalize Azumaya \cite{azumaya} by showing that every weakly nil clean ring of bounded
index of nilpotence is strongly $\pi$-regular. In particular, weakly nil clean PI-rings are strongly $\pi$-regular.
We also prove that the center of a weakly nil clean ring is (strongly) $\pi$-regular, which extends the analogous result of McCoy \cite{mccoy} for $\pi$-regular rings.
This also implies that every weakly nil clean ring is a corner of a clean ring, and thus provides a partial answer
to \cite[Question 3.10]{ster3}, which asks if this actually holds for all weakly clean rings (in the sense of \cite{ster}).

Throughout this text, rings are assumed to be unital and non-commutative. Since our definition actually needs no unity in the ring,
we will sometimes also refer to rings that do not necessarily have unity, and call them {\it non-unital} rings.
The notations $\U(R)$, $J(R)$, $\Id(R)$, $\Nil(R)$ will stand for the set of units, the Jacobson radical, the set of idempotents and the set
of nilpotents of $R$, respectively. We denote by $M_n(R)$ the ring of all $n\times n$ matrices over $R$. The letters $\integers$, $\integers_n$
and $\posintegers$ stand for the set of integers, integers modulo $n$, and positive integers, respectively.

\section{Basic properties and examples}

Our central tool is the following notion:

\begin{definition}
\label{definicija}
Let $R$ be a non-unital ring. An element $a\in R$ is \textit{weakly nil clean} in $R$ if there exist $e\in\Id(R)$ and $q\in\Nil(R)$ such that $a-e-q\in eRa$.
A non-unital ring $R$ is \textit{weakly nil clean} if every element $a\in R$ is weakly nil clean in $R$.
\end{definition}

\begin{remark}
Our notion of weakly nil clean rings should not be confused with the one arising from the notion of weakly clean rings as defined by Ahn and Anderson in \cite{ahnanderson}.
There, weakly clean rings are defined as rings in which every element is a sum or a difference of a unit and an idempotent.
Motivated by this definition, an analogous definition of weakly nil clean rings can be introduced,
as rings in which every element is a sum or a difference of a nilpotent and an idempotent
(see \cite{danchevmcgovern} and \cite{danchevdiesl}). Note that these definitions are unrelated to our definition.
Despite the possible confusion, we decided to stick with our chosen name, to keep the name as much self-explanatory as possible.
Our definition of weakly nil cleanness is built upon weakly clean rings as defined in \cite{ster}, and nil clean rings defined in Diesl \cite{diesl}.
\end{remark}

In the following we provide some basic examples. We will give more examples later when we have more tools at our disposal.
In the next section we will show that weakly nil clean rings include all $\pi$-regular rings.

\begin{example}
(1) Every idempotent and nilpotent in a ring is weakly nil clean. In particular, Boolean rings and nil (non-unital) rings are weakly nil clean.

(2) {\it Nil clean rings} are rings in which every element can be written as a sum of an idempotent and a nilpotent (see Diesl \cite{diesl}).
Clearly, every nil clean ring is weakly nil clean.

(3) Every invertible element in a ring is weakly nil clean. Indeed, if $a$ is invertible in $R$, then we can write $a=1+0+1\cdot (1-a^{-1})\cdot a$, where $1\in\Id(R)$
and $0\in\Nil(R)$. In particular, this shows that every division ring is weakly nil clean.
Note that division rings are never nil clean in the sense of Diesl \cite{diesl}, except when $R=\integers_2$.
\end{example}

\begin{proposition}
\label{osnove}
\phantom{a}\hfill
\begin{enumerate}
\item
Every weakly nil clean ring is weakly clean in the sense of \v{S}ter \cite{ster3}. In particular, every weakly nil clean ring is exchange.
\item
Weakly nil clean rings are closed under homomorphic images, direct limits and finite direct products.
\end{enumerate}
\end{proposition}

\begin{proof}
To see (i), note that one of the equivalent characterizations of the weakly clean property is \cite[Definition 4.1]{ster3}, which is clearly weaker than our
Definition \ref{definicija}. The proof of (ii) is straightforward, so we omit it.
\end{proof}

The following gives another characterization of weakly nil clean rings which is analogous to the equivalence (i)$\Leftrightarrow$(ii) of \cite[Definition 2.1]{ster3}.
This characterization also gives a direct proof that weakly nil clean rings are exchange.

\begin{proposition}
\label{prva}
A ring $R$ is weakly nil clean if and only if for every $a\in R$ there exist an idempotent $e\in Ra$ and a nilpotent $q\in R$ such that $1-e=(1-e)(1+q)(1-a)$.
\end{proposition}

\begin{proof}
Let us prove only the forward direction, the converse is proved similarly.
Let $a\in R$, and, by assumption, write $-a=e+q+exa$ where $e\in\Id(R)$, $q\in\Nil(R)$ and $x\in R$. Denote $f=1-e$ and $u=1+q\in\U(R)$.
We have $e+u-1=-a-exa\in Ra$, hence $u^{-1}eu=u^{-1}e(e+u-1)$ is an idempotent in $Ra$. Furthermore, we have $fa=f(-e-q-exa)=-fq$, hence
$fu=f(1-a)$ and thus $u^{-1}fu=u^{-1}fu\cdot u^{-1}\cdot(1-a)$. The proof is complete by observing that $u^{-1}-1\in\Nil(R)$.
\end{proof}

The following fundamental question remains open.

\begin{question}
Is the weakly nil clean property left-right symmetric?
\end{question}

Note that for weakly clean rings in the sense of \cite{ster3}, the analogous question has a positive answer (see \cite[Remark 2.4]{ster3}).
However, the methods that can be used in a direct proof of this fact, seemingly do not work in our case of weakly nil clean rings.

The following is an analogue of Diesl's result \cite[Proposition 3.15]{diesl} for nil clean rings.

\begin{proposition}
\label{nilideal}
Let $R$ be a ring with a nil ideal $I$. Then $R$ is weakly nil clean if and only if $R/I$ is so.
\end{proposition}

\begin{proof}
We only need to prove the backward direction. So, assume that $R/I$ is weakly nil clean. We write $\overline{R}=R/I$ and
$\overline{x}=x+I\in\overline{R}$ for any $x\in R$. Take $a\in R$. By assumption,
we can write $\overline{a}=\epsilon+\eta+\epsilon\chi\overline{a}$ for some $\epsilon\in\Id(\overline{R})$, $\eta\in\Nil(\overline{R})$ and $\chi\in\overline{R}$.
Since idempotents lift modulo every nil ideal, there exists $e\in\Id(R)$ with $\overline{e}=\epsilon$. Taking $x\in R$ with $\overline{x}=\chi$, we have
$\overline{a-e-exa}=\eta\in\Nil(\overline{R})$, so that $\overline{(a-e-exa)^n}=0$ for some $n\ge 1$, or equivalently, $(a-e-exa)^n\in I$.
Since $I$ is nil, it follows that $a-e-exa$ is a nilpotent, as desired.
\end{proof}

\begin{proposition}
\label{radikal}
A ring $R$ is weakly nil clean if and only if $R/J(R)$ is weakly nil clean and $J(R)$ is nil.
\end{proposition}

\begin{proof}
According to Proposition \ref{nilideal}, we only need to prove that $J(R)$ is nil whenever $R$ is weakly nil clean.
So, take $a\in J(R)$ and, by assumption, write $a=e+q+exa$ with $e\in\Id(R)$, $q\in\Nil(R)$ and $x\in R$.
Since $e+q=a-exa\in J(R)$, we have $e+J(R)=-q+J(R)$, so that this element is simultaneously an idempotent and a nilpotent in $R/J(R)$.
This forces $e\in J(R)$ and hence $e=0$, giving that $a=q$ is a nilpotent.
\end{proof}

The above results allow us to give some more examples.

\begin{example}
(1) Every ideal of a weakly nil clean ring is a weakly nil clean (non-unital) ring. More generally, an ideal $I$ of a ring $R$ is weakly nil clean if and only if
every $a\in I$ is weakly nil clean in $R$. This can be proved in a similar way as Proposition \ref{radikal}.

(2) Proposition \ref{nilideal} allows us to construct new weakly nil clean rings from known ones. For example, if $R$ is a weakly nil clean ring, then
the upper triangular matrix ring $T_n(R)$ is weakly nil clean by Proposition \ref{nilideal} (and Proposition \ref{osnove}).
Another example, if $R$ is a weakly nil clean ring and $M$ is an $R$-bimodule, then the {\it bimodule extension} of $R$ by $M$
(i.e., the direct sum $R\oplus M$, equipped with multiplication $(a,x)(b,y)=(ab,ay+xb)$) is weakly nil clean ring by Proposition \ref{nilideal}.

(3) Proposition \ref{radikal} provides a rich source of exchange (or clean) rings that are not weakly nil clean. For example, by this proposition,
a local ring is weakly nil clean if and only if its Jacobson radical is nil. Note that local rings are always clean.

(5) The ring $R(Q,L)$ in \cite[Example 1.7]{nicholson} is an example of an exchange ring with zero Jacobson radical which is not weakly nil clean.

(4) If $R$ is weakly nil clean then $R[X]/(X^n)$ is weakly nil clean for every $n$, which is another easy consequence of Proposition \ref{nilideal}.
\end{example}

\section{Relations to $\pi$-regular rings}

The following technical lemma will prove useful.

\begin{lemma}
\label{mocnalema}
Let $R$ be a ring and $a\in R$. If there exists $e=e^2=1-f\in Ra$ such that $faf$ is weakly nil clean in $fRf$,
then $a$ is weakly nil clean in $R$.
\end{lemma}

\begin{proof}
We imitate the proof of \cite[Lemma 2.6]{ster3}.
Let $a\in R$, $e=e^2=1-f\in Ra$, and suppose that $faf$ is weakly nil clean in $fRf$, i.e.,
\begin{equation}
\label{mlenacba}
faf=g+q+gxfaf
\end{equation}
for some $g\in\textnormal{Id}(fRf)$, $q\in\Nil(fRf)$, and $x\in fRf$.
Defining $\mu=q+fae$, we can directly verify that $\mu^n=q^n+q^{n-1}fae$ for every $n\ge 1$, so that $\mu$ is a nilpotent in $R$.
Let $\pi=1-(f-g)=e+g\in\Id(R)$. From (\ref{mlenacba}) we have $(f-g)faf=(f-g)q$, hence
$(1-\pi)a=(f-g)a=(f-g)fae+(f-g)faf=(f-g)fae+(f-g)q=(f-g)\mu=(1-\pi)\mu$, which gives $(1-\pi)(a-\mu)=0$. Thus, $a-\pi-\mu\in\pi R$. Furthermore,
since $e\in Ra$ and $af=a-ae\in Ra$, we have
$\pi+\mu=e+g+q+fae=e+fae+faf-gxfaf\in Ra$. Therefore $a-\pi-\mu\in \pi R\cap Ra=\pi Ra$, as desired.
\end{proof}

In \cite{ster3}, it is proved that every $\pi$-regular ring is weakly clean (in the sense of \cite{ster3}).
The following proposition strengthens this result.
We say that an element $a$ of a ring $R$ is {\it $\pi$-regular} if there exist $n\in\posintegers$ and $r\in R$ such that $a^nra^n=a^n$.

\begin{proposition}
\label{piregularni}
Every $\pi$-regular element is weakly nil clean. In particular, every $\pi$-regular ring is weakly nil clean.
\end{proposition}

\begin{proof}
Let $a\in R$ be a $\pi$-regular element. Choose $n\ge 1$ and $r\in R$ such that $a^n=a^nra^n$. Then $e=ra^n$ is an idempotent in $Ra$.
Setting $f=1-e$, we see, as in \cite[Example 2.7 (1)]{ster3}, that $faf$ is a nilpotent in $fRf$.
Since every nilpotent element is weakly nil clean, the conclusion follows by Lemma \ref{mocnalema}.
\end{proof}

This result provides many examples of weakly nil clean rings.

\begin{example}
(1) A ring $R$ is {\it (von Neumann) regular} if for every $a\in R$ we have $a\in aRa$. Every regular ring is $\pi$-regular and hence weakly nil clean.
Moreover, Artinian rings, (right or left) perfect rings, and algebraic algebras over a field are weakly nil clean because they are $\pi$-regular.

(2) A finite direct product of $\pi$-regular rings is weakly nil clean. It is not known to the authors if it is actually always $\pi$-regular, but it seems that in general,
this is not the case.

(3) An infinite direct product of weakly nil clean rings need not be weakly nil clean. For example, $\integers_2\times\integers_4\times\integers_8\times\ldots$
is not weakly nil clean since it has non-nil Jacobson radical.

(4) If $R$ is a $\pi$-regular ring, then $T_n(R)$ is weakly nil clean by Proposition \ref{nilideal}. It seems to be very unlikely (though it is not known to the authors)
that this ring be $\pi$-regular in general.

(5) A semiperfect ring is weakly nil clean if and only if the Jacobson radical of the ring is nil. This follows from Proposition \ref{nilideal} and Proposition \ref{piregularni}.

(6) It can be easily seen that Proposition \ref{piregularni} holds also for non-unital rings.
From that we deduce that for any ring $R$, the right socle of $R$, $\soc(R_R)$, is a weakly nil clean non-unital ring.
Indeed, in $\soc(R_R)$ there is a nil ideal $I$ such that $\soc(R_R)/I$ is regular (see, for example, \cite[Proof of Lemma 1.2]{baccella}).
\end{example}

Based on the example of Rowen \cite{rowen2}, we can give an example of a weakly nil clean ring which is not $\pi$-regular.

\begin{example}
\label{rowenex}
The following construction is due to Ced\'o and Rowen \cite[Example 1]{cedorowen}.
Let $F$ be a field, and $F(X)$ the field of fractions of the polynomial ring $F[X]$. Extend $\{X^n|\,n\in\integers\}$ to a basis $\mathcal{B}$ of $F(X)$ over $F$.
Let $T$ be the free product of $F(X)$ with the (unital) free algebra on two elements $F\langle A,B\rangle$. Let
$V=\{BwA|\,w\in\mathcal{B}\setminus\{X^n|\,n<0\}\}\subseteq T$, and $P$ be the ideal of $T$ generated by $A^2$, $B^2$, $AwA$, $AwB$, $BwB$, for all $w\in\mathcal{B}\setminus\{1\}$,
$V$, and $\bigcup_{k=1}^\infty\{(BX^{-1}A)^{n_k},(BX^{-2}A)^{n_k},\ldots,(BX^{-k}A)^{n_k}\}$, where $n_k>2^k+1$. Set $R=T/P$.
Then, by \cite[Theorem 5]{cedorowen}, $R$ is a local ring whose Jacobson radical is locally nilpotent.
In fact, $J(R)=R\overline{A}R+R\overline{B}R$, where $\overline{A},\overline{B}$ are homomorphic images of $A,B$ in $R$, and $R/J(R)\cong F(X)$.
Set $S=M_2(R)$. Then $S$ is semiperfect and the Jacobson radical $J(S)=M_2(J(R))$ is nil since $J(R)$ is locally nilpotent. Hence $S$ is weakly nil clean.
But this ring is not $\pi$-regular by \cite[Example 4]{kimlee}.
\end{example}

We can also find a weakly nil clean ring with trivial Jacobson radical which is not $\pi$-regular:

\begin{example}
Let $R$ be as in Example \ref{rowenex}. We embed $R$ in a $\pi$-regular ring $S$ with $J(S)=0$.
For example, considering $R$ as an algebra over $F$, we can embed $R$ in the endomorphism ring $S=\End_F(R)$ (which is von Neumann regular).
Now, the ring
$$T=\{(a_1,\ldots,a_n,a,a,\ldots)|\,n\in\posintegers,\,a_i\in S,\,a\in R\}$$
is weakly nil clean since $S$ and $R$ are weakly nil clean,
and $T$ is not $\pi$-regular since it has a homomorphic image $R$. Note that also $J(T)=0$ by direct computation.
\end{example}

In \cite[Proposition 4.9]{ster3} it is proved that if $R$ is a ring with an ideal $I$ such that $I$ and $R/I$ are both $\pi$-regular, then
$R$ is weakly clean (in the sense of \cite{ster3}). We can prove that, under these assumptions, $R$ is even weakly nil clean:

\begin{proposition}
\label{expiregular}
Let $R$ be a ring with an ideal $I$ such that $I$ and $R/I$ are both $\pi$-regular. Then $R$ is weakly nil clean.
\end{proposition}

\begin{proof}
Only a minor correction is needed in the proof of \cite[Proposition 4.9]{ster3}, in the following way. In the last paragraph of the proof it is observed
that $(1-q)a(1-q)$ is $\pi$-regular in $R$. This clearly implies that $(1-q)a(1-q)$ is also $\pi$-regular in $(1-q)R(1-q)$.
Therefore, by Proposition \ref{piregularni}, $(1-q)a(1-q)$ is weakly nil clean in $(1-q)R(1-q)$, and by Lemma \ref{mocnalema}, $a$ is weakly nil clean in $R$.
\end{proof}

\begin{remark}
Note that, in general, the ring $R$ in Proposition \ref{expiregular} need not be $\pi$-regular. Consider, for instance, Example \ref{rowenex}.
\end{remark}

In the next few propositions we show that, for a wide range of examples, Proposition \ref{piregularni} also holds in the converse direction.

Recall that a ring $R$ is \textit{abelian} if all idempotents in $R$ are central.
In the next proposition we prove that every abelian weakly nil clean ring is strongly $\pi$-regular (and thus $\pi$-regular).
Recall that one of the equivalent definitions of strongly $\pi$-regular rings is that for every $a\in R$ there exists
$e\in\Id(R)$ such that $ea=ae$, $ae$ is a unit in $eRe$ and $a(1-e)$ is a nilpotent in $(1-e)R(1-e)$
(see \cite[equivalence (1)$\Leftrightarrow$(4) on p.~3589]{nicholson3}).

\begin{proposition}
\label{abelovost}
An abelian ring is weakly nil clean if and only if it is strongly $\pi$-regular.
\end{proposition}

\begin{proof}
We only need to prove the forward direction. Thus, let $R$ be an abelian weakly nil clean ring, and take $a\in R$.
Write
\begin{equation}
\label{strenacba}
a=e+q+exa,
\end{equation}
where $e\in\Id(R)$, $q\in\Nil(R)$ and $x\in R$. Multiplying this equation by $1-e$, we get $a(1-e)=q(1-e)$, which is a nilpotent in $R(1-e)$
since $R$ is abelian. Moreover, multiplying (\ref{strenacba}) by $e$, we get $(1-x)ae=qe+e$ and thus $(1+q)^{-1}(1-x)ae=e$.
Hence $ae$ is left invertible in $Re$. But $Re$ is abelian and hence directly finite. Accordingly, $ae\in\U(Re)$, as desired.
\end{proof}

A ring $R$ is said to have {\it bounded index of nilpotence} (or {\it bounded index} for short), if there exists $n\in\posintegers$ such that $x^n=0$ for every nilpotent $x\in R$.
Azumaya \cite{azumaya} proved that every $\pi$-regular ring with bounded index is strongly $\pi$-regular.
We will extend this result to the wider class of weakly nil clean rings.

\begin{lemma}
\label{boundedindexlema}
Let $R$ be an exchange ring with bounded index, and suppose that every homomorphic image of $R$ has nil Jacobson radical.
Then $R$ is strongly $\pi$-regular.
\end{lemma}

\begin{proof}
Assume that $R$ satisfies the assumptions of the lemma. By \cite[Theorem 23.2]{tuganbaev}, it suffices to prove that every prime factor ring of $R$ is strongly $\pi$-regular.
Thus, let $P$ be a prime ideal of $R$. By Zorn's lemma, we can find a minimal prime ideal $P_0$ of $R$ with the property that $P_0\subseteq P$. By \cite[Remark 14.4 (1)]{tuganbaev},
the ring $R/P_0$ is of bounded index, hence by \cite[Proposition 14.5 (1)]{tuganbaev} and \cite[Corollary 2]{camilloyu}
(noticing that $R/P_0$ is exchange), $R/P_0$ is semiperfect.
Since $R/P_0$ is also semiprime and $J(R/P_0)$ is nil by assumption, \cite[Corollary 14.3 (2)]{tuganbaev} forces $J(R/P_0)=0$.
Thus $R/P_0$ is semisimple Artinian and hence strongly $\pi$-regular. The ring $R/P$, being a homomorphic image of $R/P_0$, is therefore also strongly $\pi$-regular, as desired.
\end{proof}

The following extends Azumaya \cite{azumaya}.

\begin{proposition}
\label{boundedindex}
Every weakly nil clean ring of bounded index is strongly $\pi$-regular.
\end{proposition}

\begin{proof}
By Proposition \ref{osnove} and Proposition \ref{radikal}, every homomorphic image of a weakly nil clean ring has nil Jacobson radical,
hence the result follows from Lemma \ref{boundedindexlema}.
\end{proof}

Note that, by Proposition \ref{abelovost}, every commutative weakly nil clean ring is strongly $\pi$-regular.
Using Proposition \ref{boundedindex} and Levitzki's result that semiprime PI-rings have bounded index (cf.~\cite[1.6.23, 1.6.26]{rowen}), this can be immediately extended to PI-rings.
Recall that $R$ is said to be a PI-ring if $R$ satisfies a polynomial identity with coefficients in $\integers$ (or equivalently, in the center) and at least one coefficient is invertible
(see \cite{rowen}).

\begin{corollary}
\label{pi}
Every weakly nil clean PI-ring is strongly $\pi$-regular.
\end{corollary}

\begin{proof}
Let $R$ be a weakly nil clean PI-ring, and let $N(R)$ denote the prime radical of $R$.
Then the factor ring $R/N(R)$ has bounded index of nilpotence by Levitzki's theorem.
Thus, Proposition \ref{boundedindex} shows that $R/N(R)$ is strongly $\pi$-regular, which in turn forces $R$ to be strongly $\pi$-regular by \cite[Theorem 2.3]{fishersnider}.
\end{proof}

The following provides a partial answer to Diesl's question \cite[Question 4]{diesl} whether every nil clean ring is strongly $\pi$-regular.

\begin{corollary}
Let $R$ be a nil clean ring which is of bounded index or a PI-ring. Then $R$ is strongly $\pi$-regular.\hfill$\qed$
\end{corollary}

\begin{remark}
Yu \cite{yu} proved that every exchange ring of bounded index has primitive factors Artinian, and
Menal \cite{menal} proved that every $\pi$-regular ring with primitive factors Artinian is strongly $\pi$-regular.
This raises the question if the `bounded index' assumption in Proposition \ref{boundedindex} can be weakened to `primitive factors Artinian'.
However, the ring $S$ in Example \ref{rowenex} is weakly nil clean with primitive factors Artinian (in fact, every semiprimitive factor of $S$ is Artinian), but $S$ is not $\pi$-regular.

We also remark that, according to \cite{chen} and \cite{huhkimlee}, every exchange ring with primitive factors Artinian is clean. But, in general, such a ring is not weakly nil clean
(consider, for example, local rings with non-nil Jacobson radical).

\end{remark}

If $e$ is an idempotent of a ring $R$, then $eRe$ is a subring of $R$ with identity $e$, called a {\it corner ring} of $R$.
It is easy to see that every corner of a (strongly) $\pi$-regular ring is (strongly) $\pi$-regular.
Thus, it is natural to ask:

\begin{question}
\label{vprkoti}
Is every corner of a weakly nil clean ring weakly nil clean?
In particular, if the matrix ring $M_n(R)$ over a ring $R$ is weakly nil clean, is then $R$ weakly nil clean?
\end{question}

Note that for weakly clean rings (in the sense of \cite{ster3}), the analogous question has a positive answer (cf.~\cite[Proposition 2.5]{ster3} and \cite[Proposition 3.3]{ster}).
On the other hand, for clean rings the answer in general is negative (cf.~\cite{ster} and \cite{ster4}).
As our definition seems to be closer to the $\pi$-regular rings, there is some hope for the above question to have a positive answer.
Unfortunately, the methods used in \cite{ster} for weakly clean rings cannot be applied for the case of weakly nil clean rings.
But we are able to give the following partial result:

\begin{proposition}
\label{kotiabelovost}
Let $R$ be an abelian ring. If $M_n(R)$ is weakly nil clean for some $n\in\posintegers$, then $R$ is weakly nil clean.
\end{proposition}

\begin{proof}
Suppose that $M_n(R)$ is weakly nil clean. We follow the idea used in the proof of \cite[Proposition 2.2]{ster}.
Let $a\in R$ and define $A=\diag(a,0,\ldots,0)\in M_n(R)$.
By Proposition \ref{prva} we can find an idempotent $E\in M_n(R)A$ and $Q\in\Nil(M_n(R))$ such that $1-E=(1-E)(1+Q)(1-A)$.
The condition $E\in M_n(R)A$ implies that $E$ is of the form
$$E=\left(\begin{array}{cccc}e&&&\\x_1&0&&\\\vdots&&\ddots&\\x_{n-1}&&&0\end{array}\right)$$
where $e=e^2\in Ra$ and $x_i\in Re$. Let $f=1-e$, and write a block decomposition
$Q=\left(\begin{smallmatrix}\alpha&\beta\\\gamma&\delta\end{smallmatrix}\right)$ where $\alpha\in R$, $\delta\in M_{n-1}(R)$,
and $\beta,\gamma$ are blocks of appropriate sizes. Comparing the first rows of $1-E$ and $(1-E)(1+Q)(1-A)$, we get $f=f(1+\alpha)(1-a)$ and $0=f\beta$.
Since $G=\diag(f,f,\ldots,f)\in M_n(R)$ is a central matrix, $GQ=\left(\begin{smallmatrix}f\alpha&0\\f\gamma&f\delta\end{smallmatrix}\right)$ is nilpotent,
and hence $q=f\alpha$ is a nilpotent in $R$. Thus $f=f(1+q)(1-a)$, so that $a$ satisfies the condition of Proposition \ref{prva}.
\end{proof}

\begin{remark}
The matrix ring $M_n(R)$ in Proposition \ref{kotiabelovost} in general need not be (strongly) $\pi$-regular, as Example \ref{rowenex} demonstrates.
\end{remark}

The reverse form of Question \ref{vprkoti} asks the following:

\begin{question}
\label{vprmatrik}
If $e$ is an idempotent of a ring $R$ such that both $eRe$ and $(1-e)R(1-e)$ are weakly nil clean rings, is then $R$ also weakly nil clean?
In particular, if $R$ is a weakly nil clean ring, is then $M_n(R)$ weakly nil clean for every $n$?
\end{question}

The analogous statement holds true for weakly clean rings (cf.~\cite[Proposition 3.3]{ster}), as well as clean rings (cf.~\cite[Lemma on p.~2590]{hannicholson}).
For the class of nil clean rings \cite{diesl}, the question is still open. However, it looks unlikely for the above question to have a positive answer,
since for strongly $\pi$-regular and $\pi$-regular rings, the answer is negative (as Example \ref{rowenex} demonstrates).
Another reason for our pessimism is that, in view of the fact that every nil ring is weakly nil clean,
a positive answer to Question \ref{vprmatrik} would also imply, as a side effect, that the K\"othe conjecture has a positive answer.
However, although finding a counterexample for our question should be easier than that for the K\"othe conjecture,
so far we have not been successful.

Note that, in some cases a positive answer to Question \ref{vprmatrik} can be given by making use of known results.
For example, if $R$ is of bounded index or a PI-ring,
then it is known that the strongly $\pi$-regular property of $R$ and of $M_n(R)$ are equivalent (see \cite[Corollary on p.~256]{tominagayamada}).
In particular, for PI-rings we have the following:

\begin{corollary}
For a PI-ring $R$ and every $n\in\posintegers$, the following are equivalent:
\begin{enumerate}
\item
$R$ is weakly nil clean.
\item
$R$ is $\pi$-regular.
\item
$R$ is strongly $\pi$-regular.
\item
$M_n(R)$ is weakly nil clean.
\item
$M_n(R)$ is $\pi$-regular.
\item
$M_n(R)$ is strongly $\pi$-regular.
\end{enumerate}
\end{corollary}

\begin{proof}
We already know that (i)--(iii) are equivalent, and since matrices over a PI-ring are a PI-ring (see, for example, \cite[Exercise 1.8]{rowen}), (iv)--(vi) are also equivalent.
Implication (vi)$\Rightarrow$(iii) holds for every ring, and (iii)$\Rightarrow$(vi) is \cite[Corollary on p.~256]{tominagayamada}.
\end{proof}

We close this section with one more property of weakly nil clean rings.
It is known that the center of every $\pi$-regular ring is (strongly) $\pi$-regular (see \cite[Theorem 1]{mccoy}).
On the other hand, the center of an exchange ring,
as well as the center of a clean ring, need not be exchange (see, for example, \cite[Proposition 2.5]{burgessraphael} and \cite[Example 2.7]{hongkimlee}).
Hence, the center of a weakly clean ring (in the sense of \cite{ster3}) also need not be clean.
But we have the following:

\begin{proposition}
The center of a weakly nil clean ring is weakly nil clean (and thus strongly $\pi$-regular).
\end{proposition}

\begin{proof}
Let $R$ be a weakly nil clean ring. Take $a\in Z(R)$, and let
\begin{equation}
\label{centereqn}
a=e+q+exa
\end{equation}
be a weakly nil clean decomposition of $a$ in $R$, where $e\in\Id(R)$, $q\in\Nil(R)$, and $x\in R$. Multiplying (\ref{centereqn}) by $e$ from the left,
we get $e(1+q)=e(1-x)a$. Therefore, $e=e(1-x)a(1+q)^{-1}=e(1-x)(1+q)^{-1}a\in Ra$,
which also gives that $e\in Ra^k$ for every $k\ge 1$ since $e$ is an idempotent and $a$ is central.
Moreover, multiplying (\ref{centereqn}) by $1-e$ from the left, we get $(1-e)a=(1-e)q$, so that $(1-e)qe=(1-e)ae=0$ and hence $(1-e)q=(1-e)q(1-e)$.
Taking $n\in\posintegers$ with $q^n=0$, we conclude that $(1-e)a^n=((1-e)a)^n=((1-e)q)^n=(1-e)q^n=0$.

Now, take any $y\in R$. Then $ey(1-e)\in Ra^ny(1-e)=Ry(1-e)a^n=0$, and similarly $(1-e)ye\in(1-e)yRa^n=(1-e)a^nyR=0$. This proves the centrality of $e$.
Now we can write
$$a=e+(1-e)a+e(a-1).$$
Since $((1-e)a)^n=0$ and $e\in Ra$, this is a weakly nil clean decomposition of $a$ in $Z(R)$.
\end{proof}

\begin{remark}
The proof of the above proposition also shows that the center of a nil clean ring must be nil clean.
Indeed, if $a=e+q$ is a nil clean decomposition of $a\in Z(R)$, with $e\in\Id(R)$ and $q\in\Nil(R)$,
then the proof shows that $e\in Z(R)$, so that also $q\in Z(R)$, meaning that $a=e+q$ is actually a nil clean decomposition of $a$ in $Z(R)$.
\end{remark}

In \cite{ster3} it is proved that every $\pi$-regular ring is a corner of a clean ring.
The following generalizes this result, and provides a partial answer to \cite[Question 3.10]{ster3},
asking whether every weakly clean ring can be viewed as a corner of some clean ring.

\begin{corollary}
Every weakly nil clean ring is a corner of a clean ring.
\end{corollary}

\begin{proof}
Every weakly nil clean ring is weakly clean in the sense of \cite{ster3}. Hence we may apply \cite[Theorem 3.5]{ster3} with $k=Z(R)$.
\end{proof}

\section{Uniqueness conditions}

In \cite{nicholsonzhou3}, Nicholson and Zhou introduced a class of clean rings with the property that a clean decomposition of every element
is {\it unique}, i.e., every element $a$ can be written as the sum of an idempotent and a unit in a unique way. These rings are called
{\it uniquely clean}. Some other variations of rings with similar uniqueness conditions were studied throughout the literature.
For example, Diesl \cite{diesl} investigated {\it uniquely nil clean} rings, which can be defined analogously.

Thus, one might ask what happens if one adds an additional uniqueness condition to our initial definition of weakly nil clean rings.
Of course, the uniqueness condition might be required on the idempotent or on the nilpotent, so that we get two different classes of rings;
any of these two might be called {\it uniquely weakly nil clean} rings.
However, the following propositions show that in both cases we get nothing new.

\begin{proposition}
\label{unq1}
For a ring $R$, the following are equivalent:
\begin{enumerate}
\item
For every $a\in R$ there exist a unique idempotent $e\in R$ and a nilpotent $q\in R$ such that $a-e-q\in eRa$.
\item
$R$ is strongly $\pi$-regular with central idempotents.
\end{enumerate}
\end{proposition}

\begin{proof}
First we prove implication (i)$\Rightarrow$(ii). By Proposition \ref{abelovost} we only need to prove that
idempotents in $R$ are central. But this is obvious, since for every $e\in\Id(R)$ and $x\in R$,
we can write two weakly nil clean decompositions, namely $e=e+0+0$ and $e=(e+ex(1-e))+(-ex(1-e))+0$.
The uniqueness of the idempotent thus forces $ex(1-e)=0$. Similarly we see that $(1-e)xe=0$, and thus $e$ is central, as desired.

To prove the converse, assume that $R$ is strongly $\pi$-regular with central idempotents, and let
\begin{equation}
\label{enacbaunq}
a=e+q+exa
\end{equation}
be any weakly nil clean decomposition of an element $a\in R$, with $e\in\Id(R)$, $q\in\Nil(R)$ and $x\in R$. Since $R$ is strongly $\pi$-regular, we can find
$g\in\Id(R)$ such that $ag\in\U(Rg)$ and $a(1-g)\in\Nil(R(1-g))$. We will prove that $e=g$, which will prove the uniqueness of $e$.
First, we multiply (\ref{enacbaunq}) by $g(1-e)$ to get $ag(1-e)=qg(1-e)$. Since $q$ is a nilpotent and idempotents are central, $ag(1-e)$ is a nilpotent, i.e., $(ag)^n(1-e)=0$ for some $n$.
Since $ag$ is invertible in $Rg$, this forces $g(1-e)=0$. Moreover, multiplying (\ref{enacbaunq}) by $(1-g)e$, we get $a(1-g)e=(1+q)(1-g)e+xa(1-g)e$, which gives
$(1+q)^{-1}(1-x)a(1-g)e=(1-g)e$. Hence $a(1-g)e$ is left invertible in $R(1-g)e$.
On the other hand, $a(1-g)e$ is nilpotent because $a(1-g)$ is so. This forces $(1-g)e=0$. Joining this with $g(1-e)=0$, we get $e=g$, as desired.
\end{proof}

Recall that a ring $R$ is \textit{strongly regular} if for every $a\in R$ there exists $r\in R$ such that $a=a^2r$.
It is well known that a ring is strongly regular if and only if it is abelian regular \cite[Theorem 3.5]{goodearl}.
It is also easy to see that a ring is strongly regular if and only if it is strongly $\pi$-regular and has no nonzero nilpotents.

\begin{proposition}
\label{unq2}
For a ring $R$, the following are equivalent:
\begin{enumerate}
\item
For every $a\in R$ there exist an idempotent $e\in R$ and a unique nilpotent $q\in R$ such that $a-e-q\in eRa$.
\item
$R$ is strongly regular.
\end{enumerate}
\end{proposition}

\begin{proof}
To prove the forward direction, first observe, similarly as in the proof of Proposition \ref{unq1}, that $R$ must have central idempotents, and hence must be
strongly $\pi$-regular by Proposition \ref{abelovost}. Now, taking any $q\in\Nil(R)$, we have two weakly nil clean decompositions of the element $1$, namely
$1=1+0+0$ and $1=1+q+1\cdot(-q)\cdot 1$. The assumption thus forces $q=0$. Hence $R$ is a strongly $\pi$-regular ring without nonzero nilpotents and thus strongly regular.
The converse is obvious because every strongly regular ring is strongly $\pi$-regular and has no nonzero nilpotents.
\end{proof}

\bibliographystyle{abbrv}
\bibliography{GeneralizingPiRegularRings}

\end{document}